\documentclass[10pt]{amsart}
\usepackage{amssymb,color}
\usepackage{mathrsfs}
\usepackage[all]{xypic}

\newcommand{\luk}{\L u\-ka\-si\-e\-w\-icz}

\newtheorem{theorem}{Theorem}[section]

\newtheorem{corollary}[theorem]{Corollary}
\newtheorem{proposition}[theorem]{Proposition}
\newtheorem{claim}{Claim}
\theoremstyle{definition}
\newtheorem{definition}[theorem]{Definition}
\newtheorem{remark}[theorem]{Remark}
\newtheorem{notation}{Notation}

\begin{document}
\title[Hyperstates of IBP$_0$-algebras]{Hyperstates of involutive MTL-algebras that satisfy $(2x)^2=2(x^2)$}
\author[Tommaso Flaminio]{Tommaso Flaminio}
\address{IIIA - CSIC,
Campus de la Universidad Aut\`onoma de Barcelona s/n
 08193 Bellaterra, Spain.}
\email{tommaso@iiia.csic.es} 

\author[Sara Ugolini]{Sara Ugolini}
\address{Department of Computer Science, University of Pisa, Italy.} 
   \email{sara.ugolini@di.unipi.it}

\date{}

\maketitle
\begin{abstract}
States of MV-algebras have been the object of intensive study and attempts of generalizations. The aim of this contribution is to provide a preliminary investigation for states of prelinear semihoops and hyperstates of algebras in the variety generated by perfect and involutive MTL-algebras (IBP$_0$-algebras for short). Grounding on a recent result showing that IBP$_0$-algebras can be constructed from a Boolean algebra, a prelinear semihoop and a suitably defined operator between them, our first investigation on states of prelinear semihoops will support and justify the notion of hyperstate for IBP$_0$-algebras and will actually show that each such map can be represented by a probability measure on its Boolean skeleton, and a state on a suitably defined abelian $\ell$-group. 
\vspace{.2cm}

\end{abstract}

\section{Motivation}\label{sec:intro}
States of MV-algebras have been introduced by Daniele Mundici in \cite{munstates} as averaging processes for truth-values of \luk\ formulas.  These are mappings  of any 
MV-algebra in the real unit interval $[0,1]$ satisfying a normalization condition and a generalized version of the usual additivity law (see \cite{FK15,Mu12} for more details). 
The states of MV-algebras are strongly connected to states of abelian $\ell$-groups \cite{Good} through Mundici's categorical equivalence 
between the category of MV-algebras with homomorphisms 
and the category of abelian $\ell$-groups with strong order unit (unital $\ell$-groups) and unit-preserving $\ell$-group homomorphisms \cite{Mun86}. Indeed, if ${\bf A}$ is an MV-algebra and ${\bf G}_{\bf A}$ is its corresponding unital $\ell$-group, then the states of ${\bf A}$ and the states of ${\bf G}_{\bf A}$ are in 1-1 correspondence. 

MV-algebraic states have been widely studied in the last years (cf. \cite{FK15} and \cite{Mu12} for a brief survey), and many attempts have been made to define states to alternative algebraic structures. In particular, the task of defining states on {\em perfect} MV-algebras has been the object of several proposals \cite{DGL,DFL14,DFFG} since the very notion of state given in \cite{munstates} trivializes when applied to these structures. Indeed, every perfect MV-algebra has only one state: the function $s$ mapping its radical,  i.e. the intersection of its maximal filters, $Rad({\bf A})$ in $1$ and its co-radical $coRad({\bf A})$ in $0$ (see \cite{DGL} and Section \ref{sec:MVStates} below for further details). 

In this contribution, mimicking the insights provided by Mundici's categorical equivalence to the study of state theory, we shall define a notion of hyperstate (i.e., hyperreal-valued state) on a wider class of algebras called {\em IBP$_0$-algebras} which properly contains perfect MV-algebras and for which we recently provided a categorical equivalence with respect to a category  whose objects are {\em prelinear-semihoop triples}, that is, systems $({\bf B}, {\bf H}, \vee_e)$ where ${\bf B}$ is a Boolean algebra, ${\bf H}$ is a prelinear semihoop and $\vee_e: B\times H\to H$ is a suitably defined map, intuitively representing the natural join between elements of $B$ and $H$. If $({\bf B}, {\bf H}, \vee_e)$ and $({\bf B}', {\bf H}', \vee_e')$ are two triples, a morphism between them is a pair $(f,g)$ where $f:{\bf B}\to{\bf B}'$ is a Boolean homomorphism, $g:{\bf H}\to{\bf H}'$ is a prelinear semihoop homomorphism, and for every $(b,c)\in B\times H$, $g(b\vee_e c)=f(b)\vee_e' g(c)$.

The definition of hyperstate that we will present and study in the following sections is grounded on the fact that Boolean algebras already possess a well-established notion of {\em state}:
 probability functions. As for prelinear semihoops, we will show in the next section that each of them has a homomorphic image (as abelian $\ell$-monoid) in an abelian $\ell$-group. Thus, taking into account the categorical equivalence between IBP$_0$-algebras and prelinear semihoop triples, our main result will prove that any hyperstate on an IBP$_0$-algebra splits into a probability measure of the Boolean skeleton and a state of the largest prelinear semihoop contained in it. As a consequence, we will prove that if the IBP$_0$-algebra actually belongs to the variety generated by perfect MV-algebras, then its hyperstates  are given by a probability measure on its Boolean skeleton and a state of a suitably defined abelian $\ell$-group. 
 
 The present paper is structured in the following way:  Section \ref{sec:MVStates} is devoted to recalling basic notions and results about abelian $\ell$-groups, MV-algebras, perfect MV-algebras and their states, while in Section \ref{sec:pre} we will prove a first result which partially extends the usual Grothendieck group construction to lattice-ordered monoids. That result and its corollary will be used in Section \ref{sec:Sprelinear} to introduce a suitable notion of state of a prelinear semihoop which, in turn, allows to introduce a notion of hyperstate of IBP$_0$-algebras in Section \ref{sec:hyperstates}. In the same Section \ref{sec:hyperstates} we will prove that every hyperstate of an IBP$_0$-algebra splits into a probability measure on its Boolean skeleton and a state of the largest prelinear semihoop contained in it. We end this paper with Section \ref{conclusion} which is devoted to concluding remarks and future work on this subject.
 
\section{Abelian $\ell$-groups, MV-algebras and their states} \label{sec:MVStates}
An {\em abelian $\ell$-group with strong unit} (or {\em unital $\ell$-group} for short) is a pair
$({\bf G}, u)$ where ${\bf G}$ is an abelian $\ell$-group (see \cite{Good}) and $u\in G$ satisfies the following requirement: for every $x\in G$ there is a natural number $n$ such that $x\leq u+\ldots+u$ where, in the previous expression, $u+\ldots+ u$ is the $n$-times sum of $u$ in ${\bf G}$, and $\leq$ denotes the lattice order of ${\bf G}$. 

A {\em state of an $\ell$-group} ${\bf G}$  is a group homomorphism $\sigma$ to the additive group ${\bf R}$ of reals which further satisfies: for all $x\geq 0$ in $G$, $\sigma(x)\geq 0$ in $\mathbb{R}$. If $({\bf G}, u)$ is a unital $\ell$-group, a state of $({\bf G}, u)$ is any state of ${\bf G}$ such that $\sigma(u)=1$ (see \cite{Good} for further details).   

{\em MV-algebras} can be introduced as those structures ${\bf A}=(A, \oplus, \neg, 0, 1)$ of type $(2,1,0,0)$ for which there exists a unital $\ell$-group $({\bf G}_A, u)$ such that $A=\{x\in G_A\mid 0\leq x\leq u\}$, $x\oplus y=(x+y)\wedge u$, $\neg x=u-x$ and $1=u$. Furthermore, for every MV-algebra ${\bf A}$, the unital $\ell$-group $({\bf G}_A, u)$ is unique. Indeed, the previous construction induces a categorical equivalence, established by Mundici's functor $\Gamma$, between the categories of unital $\ell$-groups with unit-preserving $\ell$-group homomorphisms and that of MV-algebras with MV-homomorphisms \cite{Mun86}. In particular, it is worth noticing that for every morphism $h$ in the category of unital $\ell$-groups, $\Gamma(h)$ is an MV-homomorphism that is defined by restriction.

 The latter construction suggests that we can speak about {\em states} of an MV-algebra ${\bf A}$ restricting any state of $({\bf G}_A, u)$ both in its domain, which thus becomes $A$, and its codomain that, since $1$ is a strong unit for the $\ell$-group $\mathbf{R}$, restricts to the real unit interval $[0,1]$. Indeed, by a state of ${\bf A}$ we mean any map $s: A\to [0,1]$ such that $s(1)=1$ and $s(x\oplus y)=s(x)+s(y)$ for all $x,y\in A$ such that $x+y$ (the group sum) coincides with $x\oplus y$ (the MV-sum) \cite{munstates}. 
 
Although states of MV-algebras resemble finitely additive probability measures on Boolean algebras, they are intimately related with Borel (and hence $\sigma$-additive) regular measures. Indeed, by the Kroupa-Panti Theorem \cite{Kroupa} and \cite{Panti}, for every MV-algebra ${\bf A}$ the set of its states $\mathcal{S}({\bf A})$ is in 1-1 correspondence with the set of  Borel regular measures on the compact and Hausdorff space $Max({\bf A})$ of maximal MV-filters of ${\bf A}$. Precisely, for every state $s$ of ${\bf A}$ there exists a unique Borel regular measure $\mu$ of $Max({\bf A})$ such that $s$ is the Lebesgue integral  w.r.t. $\mu$ (see also \cite{FK15,Mu12} for more details).

Every MV-algebra admits at least one state. However, there are relevant examples of MV-algebras whose unique state is trivial, i.e., it only takes Boolean values, $0$ and $1$. This is the case, for instance, of {\em perfect} MV-algebras \cite{DinLett}. Mimicking the way we used to introduce MV-algebras in general, perfect MV-algebras are, up to isomorphisms, those MV-algebras of the form $\Gamma(\mathbf{Z}\times{\bf G}, (1,0))$ where $\Gamma$ is Mundici's functor, $\mathbf{Z}$ is the $\ell$-group of integers, ${\bf G}$ is any $\ell$-group, $\times$ denotes the lexicographic product between $\ell$-groups (which is an $\ell$-group iff the first component is totally ordered \cite[Example 3]{GH}), and $(1,0)\in \mathbb{Z}\times G$ is indeed a strong unit for $\mathbf{Z}\times{\bf G}$. Again, this construction lifts to a categorical equivalences shown by Di Nola and Lettieri between perfect MV-algebras and $\ell$-groups \cite{DinLett}.  
 
An immediate consequence of the previous definition shows that every perfect MV-algebra has for domain the disjoint union $G^+\cup G^-$ for a unique $\ell$-group ${\bf G}$, where $G^+$ denotes the positive cone of ${\bf G}$, $G^-=\{-x\mid x\in G^+\}$ and $x>y$ for every $x\in G^-$ and $y\in G^+$. Furthermore,  in every perfect MV-algebra ${\bf A}$, displayed as above, $x\oplus x=1$ for every $x\in G^-$ and $y\odot y=0$ for all $y\in G^+$. Therefore, every state $s$ of ${\bf A}$ maps $G^-$ in $1$ and $G^+$ in $0$, whence $s$ is also unique.

In order to overcome this limitation and noticing that $\ell$-groups have more than one trivial state, in \cite{DFL14} the authors introduced the notion of {\em lexicographic states} for a wide class of MV-algebras which includes  perfect algebras. For any algebra  ${\bf A}$ in the variety generated by perfect MV-algebras, a lexicographic state is any map of $A$ to the  MV-algebra $\mathscr{L}(\mathbb{R})=\Gamma(\mathbf{R}\times \mathbf{R}, (1,0))$ satisfying $s(1)=1$, $s(x\oplus y)=s(x)+s(y)$ whenever $x\odot y=0$ and such that the restriction of $s$ to its maximal semisimple quotient, is a state in its usual sense.  \cite[Corollary 6.7]{DFL14} shows that the class of lexicographic states of a perfect MV-algebra $\Gamma(\mathbf{Z}\times \mathbf{G}, (1,0))$ is in one-one correspondence with the class of states of $\mathbf{G}$.   
 
 As we shall see in Section \ref{sec:Sprelinear} the previous definition is a particular case of a more general construction that we will exhibit along this paper.
 
 In what follows we shall provide preliminary results which will help us, in Section \ref{sec:Sprelinear}, to provide a reasonable notion of {\em state} for a class of algebras, called prelinear semihoops,  which play a crucial role in this paper. 
 In particular, following the same lines that we recalled in this section, our axiomatization for states of a prelinear semihoop will follow directly by Goodearl's definition of state of an $\ell$-group.
\section{From $\ell$-monoids to $\ell$-groups and hoops}\label{sec:pre}
In this section we will prove and recall some basic results we shall need in the rest of this paper. As to begin with, let us recall that a {\em lattice-ordered monoid} ({\em $\ell$-monoid} for short) is a structure ${\bf M}=(M, +, \wedge, \vee, 0)$ such that $(M, +, 0)$ is a commutative monoid, $(M, \wedge, \vee)$ is a lattice, and  the following distribution laws hold for all $x,y,z \in M$:
\begin{itemize}
\item[(D1)] $x+(y\wedge z)=(x+ y)\wedge (x+ z)$,
\item[(D2)] $x+(y\vee z)=(x+ y)\vee (x+ z)$.
\end{itemize}
The following result, that we need to reprove completely, is an extension of the usual Grothendieck group construction to the case of lattice-ordered, and in general not cancellative, monoids. Assuming cancellativity, an analogous construction has been used, for instance, in \cite[\S 2.4]{CDM00} and \cite{DinLett}.
\begin{theorem}\label{Thm:Monoid}
Let ${\bf M}=(M, +, \wedge, \vee, 0)$ be a lattice-ordered monoid. Then, there is an abelian $\ell$-group ${\bf K}({\bf M})$ and a $\ell$-monoid homomorphism $h: {\bf M}\to {\bf K}({\bf M})$ which is injective iff ${\bf M}$ is cancellative.
\end{theorem}
\begin{proof}
Starting from a commutative monoid  ${\bf M}$ it is possible to define the Grothendieck group of ${\bf M}$, namely ${\bf K}({\bf M})$ (see \cite[Chapter II]{We}), by means the following construction.
Consider the equivalence relation on the cartesian product $M \times M$ given by $(x, y) \sim (x', y')$ if there exists $z \in M$ such that $z + x + y' = z + x' + y$. 
Let $K(M) = M \times M / \mathord\sim$, and for every $[x, y], [x', y'] \in K(M)$, let $[x, y] \hat{+} [x', y'] = [x + x', y + y']$.
Let $[0,0]$ be the identity and  let the inverse of $[x, y]$ be $-[x, y] = [y, x]$. Then ${\bf K}({\bf M}) = (K(M), \hat{+}, -, [0,0])$ is an abelian group, and it satisfies the universal property: there exist a monoid homomorphism $h$ such that for any other monoid homomorphism $k: {\bf M}\to {\bf G}$ into an abelian group ${\bf G}$, there exists a unique group homomorphism $l:{\bf K}({\bf M})\to {\bf G}$ such that $k=l\circ h$. In particular,  $h(x) = [x + x, x]$ for every $x \in M$.
Moreover, the homomorphism $h$ is injective iff ${\bf M}$ is cancellative (cf. for instance \cite{B67}).

Now, we are going to prove that if ${\bf M}$ is lattice-ordered, it is possible to define on ${\bf K}({\bf M})$ an $\ell$-group structure such that the claim holds. 
First, denoting with $\leq_{M}$ the lattice order of ${\bf M}$, let 
\begin{center}
$[x_{1}, y_{1}] \leq [x_{2}, y_{2}]$ if $\,\exists z: z + x_{1} + y_{2} \leq_{M} z + y_{1} + x_{2}$.
\end{center}
It is easy to see that it is a partial order on ${\bf K}({\bf M})$, for instance let us prove transitivity.
If $[x_{1}, y_{1}] \leq [x_{2}, y_{2}]$ and $[x_{2}, y_{2}] \leq [x_{3}, y_{3}]$, then by definition $\,\exists z: z + x_{1} + y_{2} \leq_{M} z + y_{1} + x_{2}$ and $\,\exists z': z' + x_{2} + y_{3} \leq_{M} z' + y_{2} + x_{3}$. By monotonicity, $z + x_{1} + y_{2} + z' + x_{2} + y_{3}\leq_{M} z + y_{1} + x_{2} + z' + y_{2} + x_{3}$, and putting $z'' = z + z' + x_{2} + y_{2} \in M$, we have that $z'' + x_{1} + y_{3} \leq_{M} z'' + y_{1} + x_{3}$, thus $[x_{1}, y_{1}] \leq [x_{3}, y_{3}]$ and $\leq$ is transitive.

Let us now define lattice operations with respect to the order $\leq$: $$[x_{1}, y_{1}] \sqcup [x_{2}, y_{2}] = [x_{1} + x_{2}, (x_{1} + y_{2}) \land (x_{2} + y_{1})],$$ $$[x_{1}, y_{1}] \sqcap [x_{2}, y_{2}] = [(x_{1} + y_{2}) \land (x_{2} + y_{1}), y_{1} + y_{2}].$$ We prove that $\sqcup$ is the join, the proof that $\sqcap$ is the meet being similar. First we prove that it is an upper bound, that is, $[x_{1}, y_{1}], [x_{2}, y_{2}] \leq [x_{1} + x_{2}, (x_{1} + y_{2}) \land (x_{2} + y_{1})]$. Indeed:
$x_{1} + (x_{1} + y_{2}) \land (x_{2} + y_{1}) \leq x_{1} + (x_{2} + y_{1})$ and similarly 
$x_{2} + (x_{1} + y_{2}) \land (x_{2} + y_{1}) \leq x_{2} + (x_{1} + y_{2})$.
Now we shall prove that it is the least upper bound, that is, for any $[x_{3}, y_{3}] \in K(M)$, if $[x_{1}, y_{1}], [x_{2}, y_{2}] \leq [x_{3}, y_{3}]$, then $[x_{1} + x_{2}, (x_{1} + y_{2}) \land (x_{2} + y_{1})] \leq [x_{3}, y_{3}]$. By hypothesis, $\,\exists z: z + x_{1} + y_{3} \leq_{M} z + y_{1} + x_{3}$ and $\,\exists z': z + x_{2} + y_{3} \leq_{M} z' + y_{2} + x_{3}$. Thus we get
$$z + z' + x_{1} + x_{2} + y_{3} \leq_{M} z + z' + x_{3} + y_{1} + x_{2}, \mbox{ and}$$
$$z + z' + x_{1} + x_{2} + y_{3} \leq_{M} z + z' + x_{3} + y_{2} + x_{1}$$ 
Hence $(z + z' + x_{1} + x_{2} + y_{3}) \land (z + z' + x_{1} + x_{2} + y_{3}) \leq_{M} (z + z' + x_{3} + y_{1} + x_{2}) \land (z + z' + x_{3} + y_{2} + x_{1})$
and since $\land$ distributes over $+$, we obtain that $z + z' + (x_{1} + x_{2}) + y_{3} \leq_{M} z + z' + x_{3} + (y_{1} + x_{2}) \land (y_{2} + x_{1})$, which means exactly that $[x_{1} + x_{2}, (x_{1} + y_{2}) \land (x_{2} + y_{1})] \leq [x_{3}, y_{3}]$.

In order to prove that ${\bf K}({\bf M})$ with $\sqcup, \sqcap$ is an $\ell$-group, we need to show that $\hat{+}$ distributes over $\sqcup$, that is:
$$ [x_{1}, y_{1}] \hat{+} ([x_{2}, y_{2}] \sqcup [x_{3}, y_{3}]) = ([x_{1}, y_{1}] \hat{+} [x_{2}, y_{2}] ) \sqcup ([x_{1}, y_{1}] \hat{+} [x_{3}, y_{3}])$$
Now, $[x_{1}, y_{1}] \hat{+} ([x_{2}, y_{2}] \sqcup [x_{3}, y_{3}]) = [x_{1}, y_{1}] \hat{+} [x_{2} + x_{3}, (y_{2} + x_{3}) \land (y_{3} + x_{2})] = [x_{1} + x_{2} + x_{3}, y_{1} +  (y_{2} + x_{3}) \land (y_{3} + x_{2})]$, while $([x_{1}, y_{1}] \hat{+} [x_{2}, y_{2}] ) \sqcup ([x_{1}, y_{1}] \hat{+} [x_{3}, y_{3}]) = [x_{1} + x_{2}, y_{1} + y_{2}] \sqcup [x_{1} + x_{3}, y_{1} + y_{3}] = [ x_{1} + x_{2} + x_{1} + x_{3}, ( y_{1} + y_{2} + x_{1} + x_{3}) \land ( y_{1} + y_{3} + x_{1} + x_{2})]$. It is easy to see that $[ x_{1} + x_{2} + x_{1} + x_{3}, ( y_{1} + y_{2} + x_{1} + x_{3}) \land ( y_{1} + y_{3} + x_{1} + x_{2})] = [x_{1} + x_{2} + x_{3}, y_{1} +  (y_{2} + x_{3}) \land (y_{3} + x_{2})]$, since $x_{1} + x_{2} + x_{3} + x_{1} +  ( y_{1} + y_{2} + x_{1} + x_{3}) \land ( y_{1} + y_{3} + x_{2}) = x_{1} + x_{2} + x_{3} + ( x_{1} + y_{1} + y_{2} + x_{3}) \land ( y_{1} + y_{3} + x_{1} + x_{2}) $, which settles the claim.

In order to conclude the proof 
we need to prove that the monoid homomorphism $h : {\bf M} \to {\bf K} ({\bf M})$, $h(x) = [x + x, x]$, is also a lattice homomorphism. Let us prove that it respects the meet operation, the proof for the join being similar. We need to show that $h(x \land y) = h(x) \sqcap h(y)$, that is to say, $$ [(x \land y) + (x \land y), x \land y] = [x + x, x] \sqcap [y + y, y].$$
Now, $[x + x, x] \sqcap [y + y, y] = [(x + x + y) \land (y + y + x), x + y]$. Since $(x + y) + (x \land y) = (x  + y + x) \land (x + y + y)$, we have that $(x \land y) + (x \land y) + x + y = (x  + y + x) \land (x + y + y) + (x \land y)$, which proves the claim. 
%
\end{proof}
\begin{definition}[\cite{EGHM}]
An algebra ${\bf H}=(H, \cdot, \to, \wedge, 1)$ of type $(2,2,2,0)$ is a {\em semihoop} if it satisfies the following conditions:
\begin{itemize}
\item[(i)] $(H, \wedge, 1)$ is an inf-semilattice with upper bound;
\item[(ii)] $(H, \cdot, 1)$ is a commutative monoid isotonic with respect to the inf-semilattice order;
\item[(iii)] For every $c_1, c_2\in H$, $c_1\leq c_2$ iff $c_1\to c_2=1$;
\item[(iv)] For every $c_1, c_2, c_3\in H$, $(c_1\cdot c_2)\to c_3=c_1\to(c_2\to c_3)$.
\end{itemize}
Furthermore, a semihoop ${\bf H}$ is said to be  {\em prelinear}  if it satisfies:
\begin{itemize}
\item[(Pre)] $(x\to y)\to z\leq ((y\to x)\to z)\to z$. 
\end{itemize}

A prelinear semihoop is said to be a {\em basic hoop} if it satisfies:
\begin{itemize}
\item[(Div)] $x \cdot (x \to y) = y \cdot (y \to x)$. 
\end{itemize}
A basic hoop is said to be {\em cancellative} if the following holds:
\begin{itemize}
\item[(Canc)] $(x\to (x\cdot y))\to y =1$.
\end{itemize}
\end{definition}
The class of prelinear semihoops forms a variety that we will denote by $\mathbb{PSH}$, while $\mathbb{BH}$ and $\mathbb{CH}$ denote the variety of basic hoops and cancellative hoops respectively.
\begin{remark}
In a semihoop ${\bf H}$ we can always define a {\em pseudo-join} as follows: $$x \lor y = ((x \to y) \to y) \land ((y \to x) \to x),$$ and $(H, \land, \lor, 1)$ is a lattice iff $\lor$ is associative. Furthermore, in a prelinear semihoop ${\bf H}$, the pseudo-join $\lor$ is the join operation on ${\bf H}$, thus $(H, \land, \lor, 1)$ is a lattice \cite[Lemma 3.9]{EGHM}. Henceforth we will include the $\lor$ in the signature of prelinear semihoops.
\end{remark}
The following result is a direct consequence of Theorem \ref{Thm:Monoid} plus the observation that the reduct ${\bf \hat{H}}=(H, \cdot, \wedge,\vee, 1)$ of a prelinear semihoop ${\bf H}$ is an $\ell$-monoid (written in multiplicative form). In the following result and in the rest of this paper, if $h:{\bf M}\to {\bf K}({\bf M})$  is an $\ell$-monoid homomorphism, we shall denote by ${\bf J}_{h[M]}$ the $\ell$-subgroup of ${\bf K}({\bf M})$ generated by $h[M]=\{h(a)\mid a\in M\}$.
\begin{corollary}\label{-Lattice}
For every prelinear semihoop ${\bf H}$ there is an abelian $\ell$-group ${\bf K}({\bf \hat{H}})$ and a $\ell$-monoid homomorphism $h: {\bf H} \to {\bf K}({\bf \hat{H}})$ which is injective iff ${\bf H}$ is cancellative. Furthermore, for every $x\in {\bf K}({\bf \hat{H}})$, there exists a $y\in {\bf J}_{h[H]}$ such that $y\leq x$. Consequently, for every $x\in {\bf K}({\bf \hat{H}})$, there exists a $y'\in {\bf J}_{h[H]}$ such that $y'\geq x$.
\begin{proof}
The first part directly follows from Theorem \ref{Thm:Monoid}. As for the second part, let $[a,b]$ be a generic element of ${\bf K}({\bf \hat{H}})$ and let, for every $x\in \hat{H}$, $h(x)=[x\cdot x, x]$. 
Thus, $a\geq a\wedge b$ and hence $a \cdot (a\wedge b)\geq (a\wedge b)\cdot (a\wedge b)\geq (a\wedge b)\cdot (a\wedge b) \cdot b$, since in every prelinear semihoop $z\geq z\cdot k$. Thus, by definition of $\leq$ in ${\bf K}({\bf \hat{H}})$, $[a,b]\geq [(a\wedge b) \cdot (a\wedge b), a\wedge b]=h(a\wedge b)\in {\bf J}_{h[H]}$. Obviously, since every element of ${\bf J}_{h[H]}$ can be equivalently displayed as $-[c,d]$ for some $[c,d]\in {\bf J}_{h[H]}$ and since $-$ reverses the order, there is a $y\in {\bf J}_{h[H]}$ such that $[c,d]\geq y'$, whence $-[c,d]\leq -y\in {\bf J}_{h[H]}$. 
\end{proof}
\end{corollary}


\section{States of prelinear semihoops}\label{sec:Sprelinear}
In this section we will introduce states of prelinear semihoops and we will show some basic properties.
The following definition naturally arises by following the same lines that inspired the axiomatization of state of an MV-algebra (recall Section \ref{sec:MVStates}), from Corollary \ref{-Lattice} and recalling  that a state of an $\ell$-group ${\bf G}$ is a map $\sigma:G\to\mathbb{R}$ which is a group homomorphism, and such that, if $x\geq 0$, then $\sigma(x)\geq 0$. 
\begin{definition}\label{hoop-states}
A {\em state} of a prelinear semihoop ${\bf H}=(H, \cdot, \to, \wedge, \lor, 1)$ is a map $w: H\to \mathbb{R}^-$ satisfying the following conditions:
\begin{itemize}
\item[(v1)] $w(1)=0$,
\item[(v2)] $w(x\cdot y)=w(x)+w(y)$,
\item[(v3)] if $x\leq y$, then $w(x)\leq w(y)$.
\end{itemize}
\end{definition}
Given any prelinear semihoop ${\bf H}$ we denote by $\mathcal{W}({\bf H})$  the set of its states. Notice that $\mathcal{W}({\bf H})$ is not empty. Indeed, letting $x\ominus y=\min \{0, x-y\}$ on $\mathbb{R}^-$,  ${\bf R}^-=(\mathbb{R}^-, +, \ominus, \leq, 0)$ is a prelinear semihoop, and any hoop-homomorphism of ${\bf H}$ to ${\bf R}^-$ is a state.
\begin{proposition}\label{prop:statesSemihoops}
For every prelinear semihoop ${\bf H}$ and for every $w\in \mathcal{W}({\bf H})$, the following hold:
\begin{enumerate}
\item $w(x\wedge y)+w(x\vee y)=w(x)+ w(y)$,
\item if ${\bf H}$ is a basic hoop, then $w(x)+w(x\to y)=w(y)+w(y\to x)$,
\item if ${\bf H}$ is divisible, then (v3) is redundant.
\end{enumerate}
\end{proposition}
\begin{proof} (1). 
The variety $\mathbb{PSH}$ is generated by its linearly ordered members, and in every totally ordered prelinear semihoop $x\cdot y=(x\wedge y)\cdot (x\vee y)$. Thus, the latter equation holds in every ${\bf H}\in \mathbb{PSH}$. Therefore, for every $w\in\mathcal{W}({\bf H})$, $w(x)+w(y)=w(x\cdot y)=w((x\wedge y)\cdot (x\vee y))=w(x\wedge y)+ w(x\vee y)$ where the last equality follows from (v2). 
\vspace{.2cm}

\noindent (2). Immediate from {\em (Div)} and {\em (v2)}.

\vspace{.2cm}

\noindent (3). Let ${\bf H}$ be divisible and assume that $x\leq y$. Then, $x=x\wedge y=y\cdot(y\to x)$ and hence  $w(x)=w(y\cdot (y\to x))=w(y)+w(y\to x)\leq w(y)$ where the last inequality holds since, by definition, $w(y),w(y\to x)\in \mathbb{R}^-$.
\end{proof}
\begin{remark}
In any prelinear semihoop $(H, \cdot, \to, \wedge, \vee, 1)$, the reduct $(H, \wedge, \vee)$ is a distributive lattice and indeed Proposition \ref{prop:statesSemihoops} (1) above shows that states of prelinear semihoops are {\em valuations} on their lattice reduct as defined by Birkhoff in \cite{B67}.

Furthermore, if ${\bf H}$ is a basic hoop, then Proposition \ref{prop:statesSemihoops}  (2) shows that $w$ satisfies Bosbach equation $w(x)+w(x\to y)=w(y)+w(y\to x)$. Thus, every state of a basic hoop can be seen as a {\em Bosbach state} in the sense of \cite{HZX}. 
\end{remark}
For the next result, recall how the  $\ell$-group ${\bf K}({\bf \hat{H}})$ and the $\ell$-monoid homomorphism $h$ are defined in Theorem \ref{Thm:Monoid} and Corollary \ref{-Lattice}.
\begin{proposition}\label{prop:stateHoopsGropus}
For every prelinear semihoop ${\bf H}$ and every $w\in \mathcal{W}({\bf H})$,  there is a state $\sigma$ of the abelian $\ell$-group ${\bf K}({\bf \hat{H}})$ such that $w=\sigma\circ h$. Conversely, if $\sigma$ is a state of ${\bf K}({\bf \hat{H}})$, then the composition map $w=\sigma\circ h$ is a state of ${\bf H}$.
\end{proposition}
\begin{proof}
Let $h$ and ${\bf K}({\bf \hat{H}})$ as in Theorem \ref{Thm:Monoid} and Corollary \ref{-Lattice} and let ${\bf J}_{h[H]}$ be the  $\ell$-subgroup of ${\bf K}({\bf \hat{H}})$ generated by $h[H]=\{h(x)\mid x\in H\}$.  For every element $[x,y]\in J_{h[H]}$, let $\hat{\sigma}([x,y])=w(y)-w(x)\in \mathbb{R}$. 
\begin{claim}\label{claimG}
The map $\hat{\sigma}:  J_{h[H]}\to \mathbb{R}$ is a state of ${\bf J}_{h[H]}$.
\end{claim}
\begin{proof} (of the Claim). The neutral element of ${\bf K}({\bf \hat{H}})$, which obviously coincide with the neutral element of ${\bf J}_{h[H]}$, is $[1,1]$. Thus, since $w(1)=0$, $\hat{\sigma}([1,1])=0$. Moreover, for every positive element $[x,1]$ of ${\bf J}_{h[H]}$, $\hat{\sigma}([x,1])=-w(x)\in \mathbb{R}^+$, whence $\sigma$ is positive. 

It is left to show that $\hat{\sigma}$ is a group homomorphism. Let $[x_1, y_1], [x_2, y_2]\in J_{h[H]}$. Then, $\hat{\sigma}([x_1, y_1]+ [x_2, y_2])=\hat{\sigma}([x_1\cdot x_2, y_1\cdot y_2])=w(y_1)+w(y_2)-w(x_1)-w(x_2)=w(y_1)-w(x_1)+w(y_2)-w(x_2)=\hat{\sigma}([x_1, y_1])+\hat{\sigma}(x_2, y_2)$. 
\end{proof}
Turning back to the proof of Proposition \ref{prop:stateHoopsGropus}, let $\sigma:{\bf K}({\bf \hat{H}})\to \mathbb{R}$ be a state of ${\bf K}({\bf \hat{H}})$ obtained by extending $\hat{\sigma}$ from the $\ell$-subgroup ${\bf J}_{h[H]}$ of ${\bf K}({\bf \hat{H}})$. The existence of $\sigma$ is hence guaranteed by \cite[Proposition 4.2]{Good} plus the observation that every element of ${\bf K}({\bf \hat{H}})$ is bounded above by an element of ${\bf J}_{h[H]}$ (second part of Corollary \ref{-Lattice}).  

Now, since every element $x$ of $H$ is represented in $K(\hat{ H})$ as $[1,x]$, $w(x)=w(x)-w(1)=\hat{\sigma}([1,x])=\sigma([1,x])$. 

Conversely, let $\sigma:{\bf K}({\bf \hat{H}})\to \mathbb{R}$ be a state. Then $w(1)=\sigma(h(1))=\sigma(0)=0$. Moreover, for every $x,y\in H$, $w(x\cdot y)=\sigma(h(x\cdot y))=\sigma(h(x)+h(y))=\sigma(h(x))+\sigma(h(y))=w(x)+w(y)$. Finally, the monotonicity of $w$ comes from the monotonicity of $\sigma$ and $h$. 
\end{proof}

\section{States of IBP$_0$-algebras and their representation}\label{sec:hyperstates}
 MTL-algebras are bounded, commutative, integral residuated lattices ${\bf A} = (A, \cdot, \to, \land, \lor, 0,1)$ further satisfying $(x\to y)\vee (y\to x)=1$. MTL-algebras form a variety that we will denote with $\mathbb{MTL}$. In every MTL-algebra ${\bf A}$ we can define further operations and abbreviations in the following manner: $\neg x:= x\to 0$, $x\oplus y:= \neg x\to y$, $2x:= x\oplus x$, $x^2:=x\cdot x$.
{\em GMTL-algebras} are unbounded MTL-algebras and they form a variety which is term-equivalent to the variety $\mathbb{PSH}$ \cite{NEG}. 
\begin{definition}
An IBP$_0$-algebra is an MTL-algebra further satisfying the following equations:
\begin{itemize}
\item[(DL)] $(2x)^2=2(x^2)$,
\item[(Inv)] $\neg\neg x=x$.
\end{itemize}
\end{definition}
\begin{remark}
Every perfect MV-algebra \cite{DinLett} is an IBP$_0$-algebra. Indeed, the class of IBP$_0$-algebras is a variety $\mathbb{IBP}_0$ that properly contains the subvariety of $\mathbb{MV}$ generated by perfect MV-algebras. More precisely, the variety generated by perfect MV-algebras is definable, within $\mathbb{IBP}_0$, by the divisibility equation $x\wedge y=x\cdot(x\to y)$.

\end{remark}
For every IBP$_0$-algebra ${\bf A}$, let us define 
$$
\mathscr{B}({\bf A})=\{a\in A\mid a\vee \neg a=1\} \mbox{ and }\mathscr{H}({\bf A})=\{x\in A\mid x>\neg x\}.
$$
It is known (see \cite[Proposition 2.5]{AFU16}) that $\mathscr{B}({\bf A})$ and $\mathscr{H}({\bf A})$ respectively are the domains 
of the largest Boolean subalgebra of ${\bf A}$ and the domain of  the radical of $A$. In \cite{AFU16} we showed a categorical equivalence between the category of IBP$_0$-algebras with homomorphisms and a category whose objects are triples $({\bf B}, {\bf H}, \vee_e)$ where ${\bf B}$ is a Boolean algebra, ${\bf H}$ is a prelinear semihoop and $\vee_e: B\times H\to H$ is a suitably defined map, intuitively representing the natural join between elements of $B$ and $H$. If $({\bf B}, {\bf H}, \vee_e)$ and $({\bf B}', {\bf H}', \vee_e')$ are two triples, a morphism between them is a pair $(f,g)$ where $f:{\bf B}\to{\bf B}'$ is a Boolean homomorphism, $g:{\bf H}\to{\bf H}'$ is a prelinear semihoop homomorphism, and for every $(b,c)\in B\times H$, $g(b\vee_e c)=f(b)\vee_e' g(c)$.

The {\em radical} of an MTL-algebra ${\bf A}$, $Rad({\bf A})$, is the intersection of its maximal filters, while the {\em co-radical} of ${\bf A}$ is defined as $coRad({\bf A})=\{a\in A\mid \neg a\in Rad({\bf A})\}$ (see for instance \cite{CT}).

A direct consequence of the categorical equivalence is the following proposition that we are going to apply later.
\begin{proposition}[\cite{AFU16}]\label{prop:elements}
Let ${\bf A}$ be any IBP$_0$-algebra. Then the following holds:
\begin{itemize}
\item[(i)] For every $a\in A$, $a=(b_a\vee \neg c_a)\land (\neg b_a\vee c_a)$ where $b_a=\neg((\neg a^2)^2)$ belongs to $\mathscr{B}({\bf A})$ and $c_a=a\vee\neg a$ belongs to $\mathscr{H}({\bf A})$.
\item[(ii)] For every $b\in \mathscr{B}({\bf A})$ and every $c\in \mathscr{H}({\bf A})$, $b\vee c\in \mathscr{H}(\bf A)$. 
\item[(iii)] $(Rad({\bf A}),\cdot, \to, \wedge, 1)$ (where operations are obtained by restriction from those of ${\bf A}$) is a prelinear semihoop isomorphic to $\mathscr{H}({\bf A})$.   
\end{itemize}
\end{proposition}

\begin{notation}\label{not1}
(1). Let ${^*}[0,1]$ be a nontrivial ultraproduct of the real unit interval and let $\varepsilon$ be an infinitesimal in ${^*}[0,1]$. The MV-algebra $\mathscr{L}(\mathbb{R})=\Gamma(\mathbf{R}\times \mathbf{R}, (1,0))$ discussed in Section \ref{sec:MVStates} is, up to isomorphisms, the MV-subalgebra of ${^*}[0,1]_{MV}$ generated by $[0,1]\cup\{r\varepsilon \,|\, r \in \mathbb{R}\}$ (see \cite[Example 6.1]{DFL14}). Therefore, every element of $\mathscr{L}(\mathbb{R})$ can be uniquely displayed as $r+\varepsilon s$ for $r\in [0,1]$ and $s\in\mathbb{R}$. In what follows we will adopt the following notation: for every $x\in \mathscr{L}(\mathbb{R})$, $x^\circ$ and $x^\ast$ denote those unique elements of $[0,1]$ and $\mathbb{R}$ respectively, such that $x=x^\circ+\varepsilon x^\ast$. 
\vspace{.1cm}

%
\noindent (2). In every MTL-algebra ${\bf A}$ we abbreviate $\neg a\to b$ as $ a\oplus b$ and for every $n,m\in \mathbb{N}$ and every element $x\in A$, $n.a$ stands for $a\oplus\ldots\oplus a$ ($n$-times) and $a^m$ is $a\cdot\ldots\cdot a$ ($m$-times).
\end{notation}
\begin{definition}\label{def:state}
For any  IBP$_0$-algebra ${\bf A}$, we define a {\em hyperstate} of ${\bf A}$ as a map $s:A\to\mathscr{L}(\mathbb{R})$ such that:
\begin{itemize}
\item[(s1)] $s(1)=1$  and $s(0)=0$,
\item[(s2)] $s(x\oplus y)+s(x\cdot y)=s(x)+s(y)$, 
\item[(s3)] If $x\vee \neg x=1$, then  $s(x)\in [0,1]$.
\end{itemize}
\end{definition}

\begin{proposition}\label{prop1}
The following properties hold for hyperstates of  IBP$_0$-algebras:
\begin{itemize}
\item[(i)] $s(\neg x)=1-s(x)$, 
\item[(ii)] if $x\leq y$, then $s(x)\leq s(y)$,
\item[(iii)] if $x\cdot y=0$, $s(x\oplus y)=s(x) + s(y)$,
\item[(iv)] if $x\oplus y=1$, $s(x\cdot y)=s(x)\cdot s(y)$,
\item[(v)] $s(x\wedge y)+s(x\vee y)=s(x)+s(y)$,
\item[(vi)] the restriction $p$ of $s$ to $\mathscr{B}({\bf A})$ is a $[0,1]$-valued and finitely additive probability measure,
\item[(vii)] if $x\in coRad({\bf A})$, then $s(x)\in coRad(\mathscr{L}(\mathbb{R}))$. If $x\in Rad({\bf A})$, $s(x)\in  Rad(\mathscr{L}(\mathbb{R}))$,
\item[(viii)] the map $w:\mathscr{H}({\bf A})\to \mathbb{R}^-$ defined as
$$
w(x)= \frac{s(x)-1}{\varepsilon}
$$ 
is a state in the sense of Definition \ref{hoop-states}.
\end{itemize}
\end{proposition}
\begin{proof}

$(i)$. In any MTL-algebra, $x\cdot\neg x=0$ and $x\oplus \neg x=1$, thus (s2)  and (s1) imply $s(x)+s(\neg x) = s(x\oplus \neg x)+s(x\cdot \neg x)=s(1)+s(0)= 1 + 0 = 1$, whence $s(\neg x)=1-s(x)$. 
\vspace{.2cm}

\noindent$(ii)$. If $x\leq y$, then $x\cdot \neg y=0$. Thus from $(s2)$, $s(x\oplus \neg y)=s(x)+s(\neg y)=s(x)+1-s(y)\leq 1$. Thus $s(x)\leq s(y)$. 

\vspace{.2cm}

\noindent$(iii)$  and $(iv)$ are direct consequences of (s1) and (s2).
\vspace{.2cm}

\noindent$(v)$. As we already observed in the proof of Proposition \ref{prop1}, in every MTL-algebra $x\cdot y=(x\wedge y)\cdot(x\vee y)$. Analogously, in every IBP$_0$-algebra, $x\oplus y=(x\wedge y)\oplus(x\vee y)$. Thus, from (s2), $s(x)+s(y)=s((x\wedge y)\cdot(x\vee y))+ s((x\wedge y)\oplus(x\vee y))=s((x\wedge y)\cdot(x\vee y))+s(x\wedge y)+s(x\vee y)-s((x\wedge y)\cdot(x\vee y))=s(x\wedge y)+s(x\vee y)$.

\vspace{.2cm}

\noindent$(vi)$. That the restriction $p$ of $s$ to $\mathscr{B}({\bf A})$ satisfies $p(1)=1$ and $p(x\wedge y)+p(x\vee y)=p(x)+p(y)$ is ensured by (s1), (s2) together with the fact that, for all $x,y\in\mathscr{B}({\bf A})$, $x\cdot y=x\wedge y$ and $x\oplus y=x\vee y$.  Finally, that for every $x\in \mathscr{B}({\bf A})$, $p(x)\in [0,1]$ is exactly (s3).

\vspace{.2cm}

\noindent$(vii)$. Let $x\in coRad({\bf A})$. Then, for every $n\in \mathbb{N}$, $n.x\leq \neg x$ and, from $(ii)$, $s(n.x)\leq s(\neg x)$. Now, $x\cdot m.x=0$ for every $m\inÊ\mathbb{N}$, whence, in particular, $s(n.x)=n.s(x)$. Thus, $n.s(x)\leq 1-s(x)$ for every $n\in \mathbb{N}$, i.e., $s(x)\in coRad(\mathscr{L}(\mathbb{R}))$. The second part of the claim now easily follows since $x\in Rad({\bf A})$ iff $\neg x\in coRad({\bf A})$ and $\alpha\in Rad(\mathscr{L}(\mathbb{R}))$ iff $\neg \alpha\in coRad(\mathscr{L}(\mathbb{R}))$ because both ${\bf A}$ and $\mathscr{L}(\mathbb{R})$ are strongly perfect MTL-algebras.

\vspace{.2cm}

\noindent$(viii)$. As we already recalled in Section \ref{sec:pre}, $\mathscr{H}({\bf A})=Rad({\bf A})$. Thus, if $x\in \mathscr{H}({\bf A})$, from $(vii)$, $s(x)\in Rad(\mathscr{L}(\mathbb{R}))$, whence there is $r_x\in \mathbb{R}^+$ such that $s(x)=1-\varepsilon r_x$. Therefore, $w(x)=s(x)/\varepsilon-1/\varepsilon=-r_x\in \mathbb{R}^-$. It is left to prove that $w$ is a state in the sense of Definition \ref{hoop-states}. First of all, $w(1)=s(1)/\varepsilon-1/\varepsilon=0$. Moreover, if $x,y\in \mathscr{H}({\bf A})$, $x\oplus y=1$, and hence $w(x\cdot y)=(s(x\cdot y)-1)/\varepsilon=(s(x)+s(y)-2)/\varepsilon=(s(x)-1)/\varepsilon + (s(y)-1)/\varepsilon=w(x)+w(y)$. The monotonicity of $w$ easily follows from the monotonicity of $s$, $(ii)$ above. Then the claim is settled. 
\end{proof}
The next result is the main theorem of this paper and it shows that each hyperstate of an IBP$_0$ algebra  ${\bf A}$ decomposes in a probability measure on its Boolean skeleton and a state on the maximal prelinear semihoop contained in $A$.
\begin{theorem}\label{thm:main}
For every IBP$_0$-algebra ${\bf A}$ and every hyperstate $s:A\to \mathscr{L}(\mathbb{R})$ 
there are a probability measure $p:\mathscr{B}({\bf A})\to[0,1]$, a state $w\in \mathcal{W}(\mathscr{H}({\bf A}))$ and an infinitesimal $\varepsilon>0$ such that, for every $a\in A$,
$$
s(a)=p(b_a)+\varepsilon (w(\neg b_a\vee c_a)-w(b_a\vee c_a)).
$$
\end{theorem}
\begin{proof}
Let $p$ and $w$ respectively be as in Proposition \ref{prop1} (vi) and (viii). 
Let $a\in A$. Then, by Proposition \ref{prop:elements} (i), $a=(b_a\vee \neg c_a)\wedge(\neg b_a\vee c_a)$ which equals $(b_a\wedge c_a)\vee (\neg b_a\wedge \neg c_a)$. Thus, since $(b_a\wedge c_a)\wedge (\neg b_a\wedge \neg c_a)=0$,
\begin{equation}\label{eq:proof1}
\begin{array}{lll}
s(a)&=&s((b_a\wedge c_a)\vee (\neg b_a\wedge \neg c_a))\\
&=&s(b_a\wedge c_a)+ s (\neg b_a\wedge \neg c_a)\\
&=&s(b_a\wedge c_a)+ s (\neg( b_a\vee c_a))\\
&=& s(b_a\wedge c_a)+ 1- s ( b_a\vee c_a)\\
&=& s(b_a\wedge c_a)+ s(b_a)+s(\neg b_a)-s (b_a\vee c_a)
\end{array}
\end{equation}
Now, since $\neg b_a\wedge (b_a\wedge c_a)=0$, $s(\neg b_a)+s(b_a\wedge c_a)=s(\neg b_a\vee (b_a\wedge c_a))=s(\neg b_a\vee c_a)$. Therefore, from (\ref{eq:proof1}), we get
$$
\begin{array}{lll}
s(a)&=&s(b_a)+ (s(\neg b_a\vee c_a)-s(b_a\vee c_a))\\
&=&p(b_a)+\varepsilon w(\neg b_a\vee c_a)+1-\varepsilon w(b_a\vee c_a)-1\\
&=&p(b_a)+\varepsilon(w(\neg b_a\vee c_a)-w(b_a\vee c_a)).
\end{array}
$$
Thus, the claim is settled.
%
\end{proof}
The following result is hence a direct consequence of Theorem \ref{thm:main} and \cite[Corollary 4.0.5]{FK15}.
\begin{corollary}\label{cor:int}
For every IBP$_0$-algebra ${\bf A}$ and every hyperstate $s:A\to \mathscr{L}(\mathbb{R})$ 
there are a regular Borel measure $\mu_s$ on the Stone space $Max(\mathscr{B}({\bf A}))$ of $\mathscr{B}({\bf A})$, a state $w\in \mathcal{W}(\mathscr{H}({\bf A}))$ and an infinitesimal $\varepsilon>0$ such that, for every $a\in A$,
$$
s(a)=\int_{Max(\mathscr{B}({\bf A}))}(b_a)^*\; {\rm d}\mu_s+\varepsilon (w(\neg b_a\vee c_a)-w(b_a\vee c_a)),
$$
where $(b_a)^*$ denotes the characteristic function of the clopen subset of $Max(\mathscr{B}({\bf A}))$ corresponding to $b_a$ via Stone duality.
\end{corollary}
Now, let ${\bf A}$ be a IBP$_0$-algebra such that $\mathscr{H}({\bf A})$ is cancellative (i.e., ${\bf A}$ belongs to the variety of MV-algebras generated by perfect MV-algebras). Then, from Corollary \ref{-Lattice} (see also \cite{B67}), $\mathscr{H}({\bf A})$ embeds into ${\bf K}(\mathscr{H}({\bf A}))$. Therefore, the following easily holds.
\begin{corollary}
Let ${\bf A}$ be a IBP$_0$-algebra such that $\mathscr{H}({\bf A})$ is cancellative. Then, for a hyperstate $s:A\toÊ\mathscr{L}(\mathbb{R})$  there are a probability measure $p:\mathscr{B}({\bf A})\to[0,1]$ and an $\ell$-group state $\sigma:{\bf K}(\mathscr{H}({\bf A}))\to \mathbb{R}$ such that, for every $a\in A$, 
$$
s(a)=p(b_a)+\varepsilon \cdot \sigma([\neg b_a\vee c_a, b_a\vee c_a])
$$
\end{corollary}
\section{Conclusions and future work}\label{conclusion}
The present paper aims at defining a notion of state of prelinear semihoops and hyperstate of IBP$_0$-algebras. Our investigation was mainly motivated by two key observations: 
\begin{itemize}
\item[1.] First, states of MV-algebras can be regarded as those mappings that arise from applying Mundici's functor $\Gamma$ to states of unital $\ell$-groups. 
Thus, with an analogue reasoning, a notion of state (more precisely, hyperstate) of perfect MV-algebras arises applying Di Nola and Lettieri's functor.
\item[2.] Second, the variety of IBP$_0$-algebras is categorically equivalent to a category of triples made of a Boolean algebra, a prelinear semihoop and a special operation which is meant to represent, in this category, the natural algebraic {\em join}. Thus, IBP$_0$-algebras can be decomposed in a Boolean algebra and a prelinear semihoop, and a notion of state of IBP$_{0}$-algebras can be inspired by this decomposition.\end{itemize}

Therefore, we first introduce a notion of state for prelinear semihoops which, in turn, is suggested by Goodearl's definition of state of an $\ell$-group and a version of Grothendieck group construction we proved  for lattice-ordered monoids. Hyperstates of IBP$_0$-algebras are then introduced and we prove that, indeed, each hyperstate can be decomposed into a probability function and a state of a prelinear semihoop. 

In our future work we plan to deepen the methodologies applied to the present paper to both extend hyperstates to other classes of (not necessarily involutive) MTL-algebras which satisfy the equation $(2x)^2=2(x^2)$ and also to provide deeper insights for these mappings. In particular, a strengthening of Theorem \ref{thm:main} and Corollary \ref{cor:int} to provide a integral representation for hyperstates of IBP$_0$-algebras. 


\begin{thebibliography}{99}

\bibitem{AFU16}
S. Aguzzoli, T. Flaminio, S. Ugolini, Equivalences between subcategories of MTL-algebras via Boolean algebras and prelinear semihoops. {\em Journal of Logic and Computation}, DOI:10.1093/logcom/exx014,  in print.

\bibitem{B67} G. Birkhoff: \emph{Lattice Theory}, 3rd Edition, Colloquium Publications, Vol. XXV, American Mathematical Society (1967).

\bibitem{CDM00}
R. Cignoli, I. M. L. DÕOttaviano, D. Mundici, {\em Algebraic foundations of many-valued reasoning}, Kluwer, 2000.

\bibitem{CT}
R. Cignoli, A. Torrens,  Free Algebras in Varieties of Glivenko MTL-algebras Satisfying the Equation $2(x^2)=(2x)^2$. {\em Studia Logica}, 83: 157--181, 2006.

\bibitem{DFFG}
D. Diaconescu, A. Ferraioli, T. Flaminio, B. Gerla, Exploring infinitesimal  events through MV-algebras and non-Archimedean states. In {\em Proceedings of IPMU 2014}, A. Laurent et al. (Eds.), Part II, CCIS 443: 385--394, 2014.

\bibitem{DFL14}
D. Diaconescu, T. Flaminio, I. Leu\c{s}tean, Lexicographic MV-algebras and lexicographic states. {\em Fuzzy 
Sets and Systems}, 244: 63--85, 2014.

\bibitem{DGL}
 A. Di Nola, G. Georgescu, I. Leu\c{s}tean, States on perfect MV-algebras, in: V. Novak, I. Perfilieva (Eds.), {\em Discovering the World With Fuzzy Logic}, in: Stud. Fuzziness Soft Comput., vol. 57, Physica, Heidelberg, 105--125, 2000. 

\bibitem{DinLett}
A. Di Nola, A. Lettieri, Perfect MV-algebras are categorically equivalent to Abelian $\ell$-groups. {\em Studia Logica} 53(3): 417--432, 1994.

\bibitem{EGHM}
F. Esteva, L. Godo, P. H\'ajek, F. Montagna. Hoops and Fuzzy Logic. {\em Journal of Logic and Computation}, 13 (4): 532--555, 2003.

\bibitem{FK15}
T. Flaminio, T. Kroupa, States of MV-algebras
{\em Handbook of Mathematical Fuzzy Logic}, Vol III. P. Cintula, C. Ferm\"uller and C. Noguera (Eds.), Studies in Logic, Mathematical Logic and Foundations, College Publications, London, 2015.

\bibitem{GH}
A. Glass, W. Holland, {\em Lattice-Ordered Groups: Advances and Techniques},  Math. Appl., vol. 48, Springer, 1989.

\bibitem{Good}
K. R. Goodearl, {\em Partially Ordered Abelian Group with Interpolation}. AMS Math. Survey and Monographs, Vol. 20, 1986.

\bibitem{H98}
P. H\'ajek. {\em Metamathematics of Fuzzy Logics}, Kluwer Academic Publishers, Dordrecht, 1998.

\bibitem{HZX}
P. He, B. Zhao, X. Xin, States and internal states on semihoops, {\em Soft Computing} 21(11): 2941--2957, 2017.

\bibitem{Kroupa}
T. Kroupa. Every state on semisimple MV-algebra is integral. {\em Fuzzy Sets and Systems} 157 (20), 2771--2787, 2006.

\bibitem{Mun86}
D. Mundici, Interpretation of ACF*-algebras in {\L}ukasiewicz sentential calculus, {\em J. Funct. Anal.} 65 15--63, 1986.

\bibitem{munstates}
D. Mundici, Averaging the Truth-value in {\L}ukasiewicz Logic.
{\em Studia Logica} 55(1), 113--127, 1995.

\bibitem{Mu12}
 D. Mundici, {\em Advanced \luk\ calculus and MV-algebras}, Trends in Logic 35, Springer, 2011.

\bibitem{NEG}
C. Noguera, F. Esteva, J. Gispert, Perfect and bipartite IMTL-algebras and
disconnected rotations of prelinear semihoops. {\em Archive for Mathematical Logic}, 44: 869--886, 2005.

\bibitem{Panti}
G. Panti. Invariant measures on free MV-algebras.  {\em Communications in Algebra}, 36(8):2849--2861, 2009.

\bibitem{We}
C. A. Weibel, The K-book: an introduction to algebraic K-theory, {\em Graduate Studies in Mathematics} 145 AMS, 2013.
\end{thebibliography}
\end{document}